\documentclass[dvipsnames,12pt]{amsart}
\usepackage[margin=1.3in]{geometry}
\usepackage[OT2,T1]{fontenc}
\usepackage[utf8]{inputenc}
\usepackage[dvipsnames]{xcolor}
\usepackage{graphicx}
\usepackage{tikz-cd}
\usepackage{stmaryrd}
\usepackage[pagebackref]{hyperref}
\hypersetup{
  colorlinks=true,
    linkcolor=purple,
    filecolor=blue,      
    urlcolor=blue,
    citecolor=ForestGreen,
    linktoc=page}

\SetSymbolFont{stmry}{bold}{U}{stmry}{m}{n}
\usepackage[scr=rsfs,bb=pazo]{mathalfa}
\usepackage[alphabetic]{amsrefs}
\usepackage{microtype,amsmath,amsfonts,amssymb,amsthm,mathrsfs,bm}
\theoremstyle{plain}
\newtheorem{theorem}[]{Theorem}
\newtheorem{theorema}{Theorem}

\newtheorem{proposition}[theorem]{Proposition}
\newtheorem{lemma}[theorem]{Lemma}

\newtheorem{corollary}[theorem]{Corollary}
\newtheorem*{conjecture*}{Conjecture}
\newtheorem*{assumption*}{Assumption}

\theoremstyle{definition}
\newtheorem{definition}[theorem]{Definition}
\newtheorem{instance}[theorem]{Example}
\newtheorem{situation}[theorem]{}
\newtheorem{remark}[theorem]{Remark}

\usepackage{enumitem}
\numberwithin{theorem}{section}
\numberwithin{equation}{section}
\setlist[itemize]{leftmargin=2.5em}

\title[A note on periods of Calabi--Yau FCI]{A note on periods of Calabi--Yau 
fractional complete intersections}
\date{\today}

\author[Tsung-Ju~Lee]{Tsung-Ju Lee}
\email{tjlee@cmsa.fas.harvard.edu}
\address{Center of Mathematical Sciences and Applications, 
20 Garden St., Cambridge, MA 02138.}
\thanks{The work is partially supported by 
the AMS--Simons travel grant}

\begin{document}
\begin{abstract}
We prove that the GKZ \(\mathscr{D}\)-module \(\mathcal{M}_{A}^{\beta}\)
arising from Calabi--Yau fractional complete intersections 
in toric varieties is complete, i.e.,~all the solutions to 
\(\mathcal{M}_{A}^{\beta}\) are period integrals.
This particularly implies that \(\mathcal{M}_{A}^{\beta}\)
is equivalent to the Picard--Fuchs system. As an application,
we give explicit formulae of the period integrals of
Calabi--Yau threefolds coming from double covers
of \(\mathbf{P}^{3}\) branch over eight hyperplanes in general position.
\end{abstract}
\maketitle
\tableofcontents

\section{Introduction}
Mirror symmetry from physics has successfully made numerous
predictions in enumerative geometry and led to many deep conjectures 
in various branches of mathematics. Roughly, mirror symmetry predicts 
for a Calabi--Yau manifold \(Y\) there is a Calabi--Yau manifold \(Y^{\vee}\)
such \(A(Y)\cong B(Y^{\vee})\) and \(B(Y)\cong A(Y^{\vee})\).
Here \(A(Y)\), the \(A\) model of \(Y\), is taken to be
the genus zero Gromov--Witten theory and \(B(Y)\), the 
\(B\) model of \(Y\), is the complex deformation theory of \(Y\).

One can study the \(B\) model via period integrals. 
It is an old strategy to study periods by firstly
investigating the PDE system they obey, called 
the Picard--Fuchs system, then computing the solutions to the Picard--Fuchs system,
and ultimately trying to restore the periods from the solutions. However, it
is difficult to write down all equations (or even a set of generators)
in the Picard--Fuchs system. Nevertheless,
in the case of Calabi--Yau hypersurfaces or complete
intersections in toric varieties, Batyrev in \cite{1993-Batyrev-variations-of-the-mixed-hodge-
structure-of-affine-hypersurfaces-in-algebraic-tori} observed that their period integrals
satisfy a certain type of GKZ systems,
introduced by Gel'fand, Graev, Kapranov, and Zelevinsky
\cites{1989-Gelfand-Kapranov-Zelevinski-hypergeometric-functions-and-toral-manifolds,
1987-Gelfand-Graev-Zelevinski-holonomic-systems-of-equations-and-series-of-hypergeometric-type}:
the input of GKZ systems constitutes a matrix 
\(A\in\mathrm{Mat}_{d\times m}(\mathbb{Z})\) and an exponent vector \(\beta\in\mathbb{C}^{d}\)
and the output is a PDE system \(\mathcal{M}_{A}^{\beta}\)
on \(\mathbb{C}^{m}\) (cf.~Definition \ref{def:gkz-system}).
Consequently, in the present case,
the GKZ system is a subsystem of the Picard--Fuchs system
but it is far away from being full in general (we say that
the GKZ system is \emph{incomplete}).
To see this, note that the GKZ system \(\mathcal{M}_{A}^{\beta}\) here
is a regular holonomic \(\mathscr{D}\)-module 
whose holonomic rank is equal to the normalized volume of 
a polytope determined by \(A\) and
the said volume is greater than the expected dimension in
most cases.
In fact, the result by Huang \textit{et al}.~in
\cite{2016-Huang-Lian-Yau-Zhu-chain-integral-solutions-to-tautological-systems}
implies that for Calabi--Yau hypersurfaces in Fano toric manifolds
the set of solutions to \(\mathcal{M}_{A}^{\beta}\) at
any point can be identified with certain
relative homology group
and later the author and Zhang generalized 
this result to arbitrary complete intersections
\cite{2020-Lee-Zhang-a-hypergeometric-systems-and-relative-cohomology}.
Finally, there are several recipes to restore period integrals. 
Inspired by mirror symmetry,
Hosono \textit{et al}.~proposed a mechanism, also known as 
the \emph{hyperplane conjecture}, to characterize the honest period
integrals among the solutions to 
\(\mathcal{M}_{A}^{\beta}\)
\cite{1996-Hosono-Lian-Yau-gkz-generalized-hypergeometric-
systems-in-mirror-symmetry-of-calabi-yau-hypersurfaces}.
Another way is to enlarge the incomplete PDE system
by adding additional operators from the symmetries of the
ambient variety. In this
direction, Hosono \textit{et al}.~introduced the extended
GKZ system to tackle the case of Calabi--Yau hypersurfaces in toric varieites
\cite{1996-Hosono-Lian-Yau-gkz-generalized-hypergeometric-systems-in-mirror-symmetry-of-calabi-yau-hypersurfaces} and later Lian \textit{et al}.~introduced
tautological systems by bringing this idea to general manifolds with
a Lie group action \cite{2013-Lian-Song-Yau-periodic-integrals-and-tautological-systems}.
Using tautological systems,
Lian \textit{et al}.~were able to prove the hyperplane conjecture
when the ambient toric variety is \(\mathbf{P}^{n}\)
\cite{2021-Lian-Zhu-on-the-hyperplane-conjecture-for-periods-of-calabi-yau-hypersurfaces-in-pn}.

The recent work of Hosono \textit{et~al}.~shed light on mirror symmetry for
\emph{singular} Calabi--Yau varieties 
\cites{2020-Hosono-Lian-Takagi-Yau-k3-surfaces-from-configurations-of-six-lines-in-p2-and-mirror-symmetry-i,2019-Hosono-Lian-Yau-k3-surfaces-from-configurations-of-six-lines-in-p2-and-mirror-symmetry-ii-lambda-k3-functions}: they
investigated singular \(K3\) surfaces arising from 
configurations of six lines in general position on \(\mathbf{P}^{2}\) 
and found their lattice-theoretical mirror.
Based on their results, 
in \cite{2020-Hosono-Lee-Lian-Yau-mirror-symmetry-for-double-cover-calabi-yau-varieties},
we constructed pairs \((Y,Y^{\vee})\) of 
singular Calabi--Yau varieties 
from double covers of
toric manifolds branch over a divisor with
strictly normal crossings
and showed that \((Y,Y^{\vee})\) are
topological mirror pairs when the dimension is less than \(5\).
In \(3\)-fold cases, we also studied explicit examples and
verified that \(A(Y)\cong B(Y^{\vee})\). In the present case,
\(A(Y)\) is the untwisted part of the 
genus zero orbifold Gromov--Witten theory and \(B(Y^{\vee})\)
is the complex deformation of \(Y^{\vee}\)
coming from particular deformations of the branch locus.
Given these, it is natural to study the
family of such a Calabi--Yau double cover.

It is known that 
the period integrals of the family of Calabi--Yau double covers 
coming from a nef-partition
also satisfy a GKZ system
with \(A\) being the matrix associated to the dual nef-partition but 
the exponent \(\beta\) being \emph{fractional} only.
Owing to the combinatorial nature of \(A\), 
one can extend classical techniques for 
Calabi--Yau complete intersections in toric varieties to
this new class of singular Calabi--Yau varieties.
Because of the striking similarity,
such a Calabi--Yau is called a
\emph{fractional complete intersection} in
\cite{2022-Lee-Lian-Yau-on-calabi-yau-fractional-complete-intersections}.

In this paper, we investigate the \(B\) model of Calabi--Yau
fractional complete intersections. We
prove that the GKZ system associated with 
a family of Calabi--Yau fractional complete intersections is equivalent to 
the Picard--Fuchs system. 
\begin{theorema}[=Theorem \ref{thm:main-theorem}]
\label{thm:theorem-a}
Let \(\mathcal{M}_{A}^{\beta}\) be the GKZ system
associated with the family of Calabi--Yau double covers
coming from a nef-partition. Then \(\mathcal{M}_{A}^{\beta}\)
is complete, i.e.,~the (classical) solutions
to \(\mathcal{M}_{A}^{\beta}\) are precisely the period integrals.
\end{theorema}
See \S\ref{subsec:cy-fci}
for the explicit construction of the Calabi--Yau double covers
from a nef-partition and \S\ref{sit:notation} for the definition of \(A\) and \(\beta\)
in the present situation.

Theorem \ref{thm:theorem-a} says there is no need
to enlarge the GKZ system although the base toric variety might
have non-torus automorphisms. 
This is the case since we only consider particular complex deformations;
the branch locus of our Calabi--Yau double cover is a divisor whose support is
a union of toric divisors and some numerically effective divisors
and we deform those numerically effective divisors only. 
The infinitesimal automorphisms fixing the union of
toric divisor are all from the maximal torus.
This gives a conceptual explanation of Theorem \ref{thm:theorem-a}.

Combined with the results in 
\cite{2022-Lee-Lian-Yau-on-calabi-yau-fractional-complete-intersections}, we
obtain all period integrals via the generalized Frobenius method 
(cf.~Corollary \ref{cor:periods-power-series}).
As an application, we give an explicit formula for
the period integrals of the family of Calabi--Yau threefolds coming
from double covers of \(\mathbf{P}^{3}\) branch over eight hyperplanes
in general position (cf.~Example~\ref{ex:p3-cover}).

Our proof of Theorem \ref{thm:theorem-a}
heavily relies on a classical result of
Gel'fand, Kapranov, and Zelevinsky.
Indeed, in the case of Calabi--Yau double covers coming
from a nef-partition, the associated GKZ system \(\mathcal{M}_{A}^{\beta}\) 
is rather special: the exponent \(\beta\) is always 
\emph{non-resonant}
with respect to \(A\) (cf.~Definition 
\ref{def:non-resonant}) and 
Theorem \ref{thm:theorem-a} follows from the main result in
\cite{1990-Gelfand-Kapranov-Zelevinsky-generalized-
euler-integrals-and-a-hypergeometric-functions}. 
The proof of the non-resonance of \(\beta\) relies on 
the combinatorial structure of \(A\) which boils down to
the convexity of the support function of numerically effective toric divisors.

We remark that for general cyclic covers of toric varieties, 
not necessarily coming from a nef-partition, 
the exponent \(\beta\) might be resonant. We will deal 
with the case when \(\beta\) is \emph{semi-nonresonant} in a forthcoming paper
\cite{2022-Lee-Zhang-a-hypergeometric-systems-and-relative-cohomology-ii}. 

\subsection*{Acknowledgments}
The author thanks 
Bong H.~Lian, Hui-Wen Lin, Chin-Lung Wang, 
and Shing-Tung~Yau for their constant encouragement,
their interests in this work, and providing him many useful comments.
He thanks Dingxin Zhang for many
valuable discussions. 
He also would like to thank 
CMSA at Harvard for hospitality while working on this project.

\section{Calabi--Yau fractional complete intersections}
\subsection{Toric varieties}
Let \(N=\mathbb{Z}^{n}\) be a rank \(n\) lattice
and \(M:=\operatorname{Hom}_{\mathbb{Z}}(N,\mathbb{Z})\) be the dual lattice.
Denote by \(\langle -,-\rangle\) the 
canonical dual pairing between \(M\) and \(N\).
Recall that a \emph{fan} \(\Sigma\) in \(N_{\mathbb{R}}:=N\otimes_{\mathbb{Z}}\mathbb{R}\)
is a collection of strongly convex rational polyhedral cones in \(N_{\mathbb{R}}\)
such that
\begin{itemize}
\item if \(\sigma\in \Sigma\), then all the faces of \(\sigma\) belong to \(\Sigma\);
\item the intersection of two cones is a face of each.
\end{itemize}
The support \(|\Sigma|\) of a fan \(\Sigma\) is the union of all 
the cones belonging to \(\Sigma\). A fan \(\Sigma\) is called
\emph{proper} if \(|\Sigma|=N_{\mathbb{R}}\). A fan \(\Sigma\)
is called \emph{simplicial} if every cone in \(\Sigma\) is simplicial, i.e.,~
generated by a \(\mathbb{Q}\) linearly independent subset.
The toric variety
defined by \(\Sigma\) is denoted by \(X_{\Sigma}\).

Recall that a \emph{polyhedron} is a finite intersection of
closed half-spaces in an Euclidean space.
A polyhedron is called a \emph{polytope} if it is bounded.
A \emph{lattice polytope}
is a polytope whose vertices are integral. Let \(\Delta\) be a polytope in 
\(\mathbb{R}^{n}\).
The dimension of \(\Delta\) is the smallest integer \(d\) such that 
\(\Delta\subset\mathbb{R}^{d}\). \(\Delta\) is called full-dimensional 
if \(d=n\). If \(\Delta\) is a full-dimensional lattice
polytope in \(M_{\mathbb{R}}:=M\otimes_{\mathbb{Z}}\mathbb{R}\)
and \(\mathbf{0}\) is its interior point, we can define the 
\emph{dual polytope} or the \emph{polar polytope} 
\begin{equation*}
\Delta^{\vee}:=\{u\in N_{\mathbb{R}}~|~\langle m,u\rangle \ge -1~\mbox{for all}~m\in \Delta\}.
\end{equation*}
A lattice polytope \(\Delta\) containing \(\mathbf{0}\) in its interior is called 
a \emph{reflexive polytope} if \(\Delta^{\vee}\) is again a lattice polytope.
One can check that \((\Delta^{\vee})^{\vee}=\Delta\) for
\(\Delta\) being reflexive.

Given a lattice polytope \(\Delta\subset M_{\mathbb{R}}\), one
can define the \emph{normal fan} \(\mathcal{N}(\Delta)\) of \(\Delta\). If 
\(\mathbf{0}\in\operatorname{int}(\Delta)\), one can define
the \emph{face fan} \(\mathcal{F}(\Delta)\) of \(\Delta\).
If \(\Delta\) is a reflexive polytope, one can easily check 
\begin{equation*}
\mathcal{N}(\Delta) = \mathcal{F}(\Delta^{\vee})~\mbox{and}~
\mathcal{N}(\Delta^{\vee}) = \mathcal{F}(\Delta).
\end{equation*}

If \(\Delta\subset M_{\mathbb{R}}\)
a lattice polytope, we denote by \(\mathbf{P}_{\Delta}\) the toric
variety defined by \(\mathcal{N}(\Delta)\). According to 
our notation, we have \(\mathbf{P}_{\Delta}=X_{\mathcal{N}(\Delta)}\).

Let \(\Sigma\) be a fan in \(N_{\mathbb{R}}\).
Denote by \(\Sigma(k)\) the set of \(k\)-dimensional cones in \(\Sigma\).
Each \(\rho\in\Sigma(1)\) determines a torus invariant divisor \(D_{\rho}\) on \(X\).
By abuse of notation, the same notation \(\rho\) also
stands for the primitive generator of the \(1\)-cone \(\rho\)
and \(\Sigma(1)\) is also regarded as the set of primitive generators
of \(1\)-cones in \(\Sigma\). Thus the notation \(\rho\in\Sigma(1)\) has two meanings but
we will not explicitly spell them out when the context is clear.
For a torus invariant divisor \(D=\sum_{\rho\in\Sigma(1)} a_{\rho}D_{\rho}\), 
we define the \emph{divisor polyhedron}
\begin{equation*}
\Delta_{D}:=\{m\in M_{\mathbb{R}}~|~\langle m,\rho\rangle\ge 
-a_{\rho}~\mbox{for all}~\rho\in\Sigma(1)\}.
\end{equation*}
Note that \(\Delta_{D}\) is a polytope when \(\Sigma\) is proper.
In which case \(\Delta_{D}\) is called the divisor polytope
and the set of integral points in \(\Delta_{D}\) gives a canonical basis of
\(\mathrm{H}^{0}(X,\mathscr{O}_{X}(D))\).
If moreover \(D\) is Cartier, for \(\sigma\in\Sigma(n)\), one can find uniquely an
\(m_{\sigma}\in M\) such that
\begin{equation*}
\langle m_{\sigma},\rho\rangle =-a_{\rho}~\mbox{if}~\rho\in 
\Sigma(1)~\mbox{and}~\rho\subset\sigma.
\end{equation*}
Suppose that every maximal cone in \(\Sigma\) is \(n\)-dimensional. 
Then a Cartier divisor \(D\) determines a collection 
\(\{m_{\sigma}\}_{\sigma\in\Sigma(n)}\) which is called the \emph{Cartier data} of \(D\).
One can define the \emph{support function of \(D\)}
\begin{equation*}
\varphi_{D}\colon |\Sigma|\to \mathbb{R}~\mbox{via}~
\left.\varphi_{D}\right|_{\sigma}(u):=\langle m_{\sigma},u\rangle.
\end{equation*}
Now assume further that \(|\Sigma|\) is convex. 
It is known that
\begin{equation*}
\mbox{\(D\) is basepoint free}~\Leftrightarrow~
\mbox{\(D\) is numerically effective}~\Leftrightarrow~
\mbox{\(\varphi_{D}\) is convex}.
\end{equation*}

\subsection{Batyrev--Borisov's duality construction}
Let \(\Delta\subset M_{\mathbb{R}}\) be a reflexive polytope 
and \(\mathcal{N}(\Delta)\) be the normal fan of \(\Delta\).
\begin{definition}
A partition \(I_{1}\sqcup\cdots \sqcup I_{r}=\mathcal{N}(\Delta)(1)\) is called
a \emph{nef-partition} if for each \(1\le j\le r\) the divisor
\begin{equation*}
E_{j}:=\sum_{\rho\in I_{j}} D_{\rho}
\end{equation*}
is Cartier and numerically effective on \(\mathbf{P}_{\Delta}\).
Note that \(E_{1}+\cdots+E_{r}=-K_{\mathbf{P}_{\Delta}}\).
A nef-partition also induces 
a decomposition 
\begin{equation*}
\Delta = \Delta_{1}+\cdots+\Delta_{r}~\mbox{where}~\Delta_{i}:=\Delta_{E_{i}}.
\end{equation*}
By abuse of the terminology, both
\(E_{1}+\cdots+E_{r}=-K_{\mathbf{P}_{\Delta}}\) and \(\Delta=\Delta_{1}+\cdots+\Delta_{r}\)
are also called nef-partitions.
\end{definition}

A nef-partition \(I_{1}\sqcup\cdots \sqcup I_{r}=\mathcal{N}(\Delta)(1)\) 
gives rise to polytopes
\begin{equation*}
\nabla_{j}:=\operatorname{Conv}(I_{j}\cup\{\mathbf{0}\}).
\end{equation*}
By definition, we have
\(\Delta^{\vee} = \operatorname{Conv}(\nabla_{1},\ldots,\nabla_{r})\).
A fundamental result in \cite{1993-Borisov-towards-the-mirror-
symmetry-for-calabi-yau-complete-intersections-in-gorenstein-toric-fano-varieties} 
states that their Minkowski sum
\begin{equation}
\label{eq:dual-nef-partition}
\nabla = \nabla_{1}+\cdots+\nabla_{r}
\end{equation}
is again a reflexive polytope and
\eqref{eq:dual-nef-partition} induces a nef-partition. This is
called the \emph{dual nef-partition} in 
\cite{1996-Batyrev-Borisov-on-calabi-yau-complete-intersections-in-toric-varieties}.

\subsection{Calabi--Yau fractional complete intersections}
\label{subsec:cy-fci}
Let \(\Delta\subset M_{\mathbb{R}}\) be a reflexive polytope.
There exists a maximal projective
crepant partial desingularization
(MPCP desingularization for short hereafter) of \(\mathbf{P}_{\Delta}\).
Recall that a MPCP desingularization
of \(\mathbf{P}_{\Delta}\) is a projective toric variety \(X_{\Sigma}\)
such that
\begin{itemize}
\item \(\Sigma\) is a refinement of \(\mathcal{N}(\Delta)\);
\item \(\Sigma\) is simplicial;
\item \(\Sigma(1)=\Delta^{\vee}\cap N\setminus \{\mathbf{0}\}\). 
\end{itemize}
Note that MPCP desingularizations exist and may not be unique. 

Given a nef-partition \(E_{1}+\cdots+E_{r}=-K_{\mathbf{P}_{\Delta}}\)
and a MPCP desingularization \(X_{\Sigma}\to\mathbf{P}_{\Delta}\),
the pullback of the nef-partition is again a nef-partition on \(X_{\Sigma}\).
By abuse of notation, it is also
denoted by \(E_{1}+\cdots+E_{r}=-K_{X_{\Sigma}}\).

Let \(S_{1}\sqcup\cdots\sqcup S_{r}=\mathcal{N}(\nabla)(1)\) be a nef-partition 
representing the dual nef-partition \(\nabla=\nabla_{1}+\cdots+\nabla_{r}\).
Then the divisor polytope of \(F_{j}=\sum_{\rho\in S_{j}} D_{\rho}\) is \(\nabla_{j}\).
By the duality construction, we have 
\begin{equation*}
\Delta_{i} = \operatorname{Conv}(S_{i}\cup\{\mathbf{0}\})~\mbox{and}~
\nabla^{\vee} = \operatorname{Conv}(\Delta_{1},\ldots,\Delta_{r}).
\end{equation*}

We recall the construction of the Calabi--Yau
double covers introduced in \cite{2020-Hosono-Lee-Lian-Yau-mirror-symmetry-for-double-cover-calabi-yau-varieties}.
We make the following assumption.
\begin{assumption*}
Both \(\mathbf{P}_{\Delta}\) and \(\mathbf{P}_{\nabla}\) admit 
a smooth MPCP desingularization.
\end{assumption*}
Let \(X\to\mathbf{P}_{\Delta}\) and \(X^{\vee}\to\mathbf{P}_{\nabla}\)
be smooth MPCP desingularizations. Let
\(E_{1}+\cdots+E_{r}=-K_{X}\) and \(F_{1}+\cdots+F_{r}=-K_{X^{\vee}}\)
be nef-partitions on \(X\) and \(X^{\vee}\) respectively.

Let \(s_{j,1}\in \mathrm{H}^{0}(X^{\vee},\mathscr{O}(F_{j}))\)
be the global section corresponding to \(\mathbf{0}\in\nabla_{j}\)
and \(s_{j,2}\in \mathrm{H}^{0}(X^{\vee},\mathscr{O}(F_{j}))\)
be a general section such that their product
\begin{equation*}
s:=\prod_{j=1}^{r}s_{j,1}s_{j,2}\in
\mathrm{H}^{0}(X^{\vee},\omega_{X^{\vee}}^{-2})
\end{equation*}
is a divisor with strictly normal crossings on \(X^{\vee}\).
Then \(s\) determines a double cover \(Y^{\vee}\) of \(X^{\vee}\)
branch over \(\{s=0\}\). Moreover, \(Y^{\vee}\) is Calabi--Yau, i.e.,~its 
canonical bundle is trivial and \(\mathrm{H}^{i}(Y^{\vee},\mathscr{O}_{Y^{\vee}})=0\)
for \(1\le i\le n-1\). 
\begin{definition}
\label{def:cy-double-cover}
The double cover \(Y^{\vee}\to X^{\vee}\)
is called \emph{the Calabi--Yau double cover
from a nef-partition}, or a
\emph{Calabi--Yau fractional complete intersection}.
\end{definition}
Deforming \(s_{j,2}\) yields a family
of singular Calabi--Yau double covers over \(X^{\vee}\)
\begin{equation*}
\mathcal{Y}^{\vee}\to V\subset 
\mathrm{H}^{0}(X^{\vee},\mathscr{O}(F_{1}))\times\cdots\times
\mathrm{H}^{0}(X^{\vee},\mathscr{O}(F_{r})).
\end{equation*}
We can apply the above construction to \(X\) which gives rise to
another family of singular Calabi--Yau double covers over \(X\)
\begin{equation*}
\mathcal{Y}\to U\subset 
\mathrm{H}^{0}(X,\mathscr{O}(E_{1}))\times\cdots\times
\mathrm{H}^{0}(X,\mathscr{O}(E_{r})).
\end{equation*}
It is conjectured in 
\cite{2020-Hosono-Lee-Lian-Yau-mirror-symmetry-for-double-cover-calabi-yau-varieties} that 
\begin{conjecture*}
\(Y\) and \(Y^{\vee}\) are mirror.
\end{conjecture*}
In this note, we will focus on their \(B\) model, i.e., the families 
\(\mathcal{Y}\to U\) and \(\mathcal{Y}^{\vee}\to V\).

\section{GKZ \texorpdfstring{\(A\)}{A}-hypergeometric systems}
\label{sec:gkz}
We can study the \(B\) model of the family \(\mathcal{Y}^{\vee}\to V\) via period integrals
and it is known that those
period integrals are governed by a certain type of GKZ systems.

To begin, let us recall the definition of GKZ systems
and the notion of non-resonance.
\begin{definition}
\label{def:gkz-system}
Let \(A=(a_{ij})\in\mathrm{Mat}_{d\times m}(\mathbb{Z})\) be an integral matrix
and \(\beta=(\beta_{i})\in\mathbb{C}^{d}\).
The \emph{GKZ system associated with \(A\) and \(\beta\)} 
is the cyclic \(\mathscr{D}\)-module
\begin{equation*}
\mathscr{D}\slash\mathscr{I}
\end{equation*}
where \(\mathscr{D}=\mathbb{C}[x_{1},\ldots,x_{m},\partial_{1},\ldots,\partial_{m}]\) with
\(\partial_{j}\equiv\partial\slash \partial x_{j}\) is the Weyl algebra on an
affine space \(\mathbb{C}^{m}\) with coordinates \(x_{1},\ldots,x_{m}\) and
\(\mathscr{I}\) is the left ideal generated by 
\begin{itemize}
\item \(\partial^{\nu_{+}}-\partial^{\nu_{-}}\), where
\(\nu_{\pm}\in\mathbb{Z}_{\ge 0}^{m}\) such that \(A\nu_{+}=A\nu_{-}\);
\item \(\sum_{j=1}^{m} a_{ij}x_{j}\partial_{j}-\beta_{i}\) for 
\(i=1,\ldots,d\).
\end{itemize}
\end{definition}

Let \(A\in\mathrm{Mat}_{d\times m}(\mathbb{Z})\)
be homogeneous, i.e., the column vectors of \(A\)
are contained in a hyperplane in \(\mathbb{R}^{d}\), and \(\mathbb{R}_{+}A\) 
denote the cone generated by 
columns of \(A\) in \(\mathbb{R}^{d}\). 
Assume that the columns of \(A\)
generate \(\mathbb{Z}^{d}\) as an abelian group.
\begin{definition}
\label{def:non-resonant}
A parameter \(\beta\in\mathbb{C}^{d}\) is called 
\emph{non-resonant with respect to \(A\)} if 
\begin{equation}
\beta\notin \bigcup_{F\subset \mathbb{R}_{+}A} 
\left(\mathbb{C}F+\mathbb{Z}^{d}\right)
\end{equation}
where the union runs through all \emph{proper} faces of \(\mathbb{R}_{+}A\). 
It is also clear that it suffices to take the union over all facets of 
\(\mathbb{R}_{+}A\).
\end{definition}

We will be only interested in the case when \(A\) and \(\beta\)
are of special types. To this end, let us fix the following notation for the
rest of the note.
\begin{situation}{\bf Notation}.
\label{sit:notation}
Suppose we are given the data in \S\ref{subsec:cy-fci}
and let notation be the same as there.
\begin{enumerate}
\item[(1)] Let \(X=X_{\Sigma}\) be a smooth
MPCP desingularization of \(\mathbf{P}_{\Delta}\)
and \(E_{1}+\cdots+E_{r}=-K_{X}\) be the pullback of the nef-partition
\(I_{1}\sqcup\cdots\sqcup I_{r}=\mathcal{N}(\Delta)(1)\).
We denote by \(J_{1}\sqcup\cdots\sqcup J_{r}=\Sigma(1)\) 
the corresponding nef-partition on \(X\). 
Let \(p=|\Sigma(1)|\). 
We write \(J_{k}=\{\rho_{k,1},\ldots,\rho_{k,m_{k}}\}\)
and put \(\rho_{k,0}:=\mathbf{0}\in N\) for each \(1\le k\le r\). 
Then \(p=m_{1}+\cdots+m_{r}\).
\item[(2)] 
Regard \(x_{k,j}\), \(0\le j\le m_{k}\), as coordinate functions on 
the affine space \(W_{k}^{\vee}:=\mathbb{C}^{m_{i}+1}\). 
Denote by \((t_{1},\ldots,t_{n})\)
the coordinate on \(T\subset X^{\vee}\),
the maximal torus of \(X^{\vee}\). Consider 
\begin{equation*}
s_{k,2} = \sum_{j=0}^{m_{k}} x_{k,j} t^{\rho_{k,j}}.
\end{equation*}
For \(x=(s_{1,2},\ldots,s_{r,2})\in 
W_{1}^{\vee}\times\cdots\times W_{r}^{\vee}\), 
we denote by \(U_{x}\) the region \(T\setminus \bigcup_{k=1}^{r} \{s_{k,2}=0\}\).
Let \(\mathscr{E}_{x}\) be the local system
on \(U_{x}\) whose monodromy exponent around \(\{s_{k,2}=0\}\) is 
\(1/2\) for each \(k\).

We will justify the notation later, i.e.,
\(s_{k,2}\) can be regard as an element in 
\(\mathrm{H}^{0}(X^{\vee},\mathscr{O}(F_{k}))\)
(cf.~Corollary \ref{cor:sec-notation}).
\item[(3)] Let \(\{e_{1},\ldots,e_{r}\}\) be the standard basis of \(\mathbb{R}^{r}\).
Set \(\mu_{i,j}:=(e_{i},\rho_{i,j})\in\mathbb{Z}^{r}\times N\).
Define matrices
\begin{equation*}
A_{i}:=
\begin{bmatrix}
\vline height 1ex & & \vline height 1ex \\
\mu_{i,0} & \cdots & \mu_{i,m_{i}}\\
\vline height 1ex & & \vline height 1ex 
\end{bmatrix}\in \mathrm{Mat}_{(r+n)\times (m_{i}+1)}(\mathbb{Z})
\end{equation*}
and 
\begin{equation*}
A:= \begin{bmatrix}
A_{1} & \cdots & A_{r}
\end{bmatrix}\in\mathrm{Mat}_{(r+n)\times (r+p)}(\mathbb{Z}).
\end{equation*}
For convenience, we will label the
columns of \(A\) by a double index \((i,j)\), i.e.,~the
\((i,j)\textsuperscript{th}\) column of \(A\) is
exactly the vector \(\mu_{i,j}\). Let 
\begin{equation*}
\beta = \begin{bmatrix}
-1/2\\
\vdots\\
-1/2\\
\vline height 1ex\\
\mathbf{0}\\
\vline height 1ex
\end{bmatrix}\in \mathbb{Q}^{r+n}.
\end{equation*}
We remark that by the smoothness of \(X\),
the columns of \(A\) generate \(\mathbb{Z}^{r+n}\)
as an abelian group.
\end{enumerate}

Throughout this note, we will be only interested in the
matrix \(A\) and the parameter \(\beta\)
defined in (3) above (\(d=r+n\) and \(m=r+p\)
according to our notation).
In this case, 
since \(A\) is homogeneous
and contains \((1,\ldots,1)\) in its row span, 
\(\mathcal{M}_{A}^{\beta}\) is 
regular holonomic.
\end{situation}

From now on, let \(A\) and \(\beta\) be 
the matrix and the parameter defined in \S\ref{sit:notation} (3).

\begin{definition}
\label{def:affine-perid}
We define \emph{affine period integrals} to be
\begin{equation}
\label{eq:affine-period}
\Pi_{\gamma}(x):=\int_{\gamma} 
\frac{1}{s_{1,2}^{1/2}\cdots s_{r,2}^{1/2}}\frac{\mathrm{d}t_{1}}{t_{1}}\wedge\cdots
\wedge\frac{\mathrm{d}t_{n}}{t_{n}},
\end{equation}
where \(\gamma\in\mathrm{H}_{n}(U_{x},\mathscr{E}_{x})\)
and \(s_{k,2} = \sum_{j=0}^{m_{k}} x_{k,j}t^{\rho_{k,j}}\in 
W_{i}^{\vee}\).
\end{definition}
It is easy to check that the
affine period integrals \eqref{eq:affine-period}
are solutions to \(\mathcal{M}_{A}^{\beta}\).
We have the chain-integral map
\begin{equation}
\label{eq:chain-integral-map}
\mathrm{H}_{n}(U_{x},\mathscr{E}_{x})\to 
\operatorname{Sol}^{0}(\mathcal{M}_{A}^{\beta})_{x},~
\gamma\mapsto \Pi_{\gamma}(x)
\end{equation}
Here \(\operatorname{Sol}^{0}(\mathcal{M}_{A}^{\beta}):=
R^{0}\mathcal{H}om_{\mathscr{D}_{W^{\vee}}}(\mathcal{M}_{A}^{\beta},\mathscr{O}_{W^{\vee}})\)
is the classical solution functor.

Now we can state our main result in this note.
\begin{theorem}
\label{thm:main-theorem}
The morphism \eqref{eq:chain-integral-map} is an isomorphism.
\end{theorem}

Recall the following result 
in \cite{1990-Gelfand-Kapranov-Zelevinsky-generalized-euler-integrals-
and-a-hypergeometric-functions}.
\begin{theorem}[\cite{1990-Gelfand-Kapranov-Zelevinsky-generalized-euler-integrals-
and-a-hypergeometric-functions}*{\S2.10}]
\label{thm:gkz-non-resonant}
Let \(A\) and \(\beta\) be as in 
\S\ref{sit:notation} (3).
If \(\beta\) is non-resonant, i.e.,
\begin{equation*}
\beta\notin \bigcup_{F\subset \mathbb{R}_{+}A} 
\left(\mathbb{C}F+\mathbb{Z}^{r+n}\right)
\end{equation*}
where the union runs through all the proper faces of \(\mathbb{R}_{+}A\),
then the map \eqref{eq:chain-integral-map}
is an isomorphism for every \(x\).
\end{theorem}


In the rest of this section, we explain how 
\(s_{k,2}\) defined in \S\ref{sit:notation} (2)
is related to an element in \(\mathrm{H}^{0}(X^{\vee},\mathscr{O}(F_{k}))\)
and therefore justify the notation \(s_{k,2}\).

Let \(\nu\in N\cap \Delta^{\vee}\) be an integral point such 
that \(\mathbb{R}_{+}\nu\) does not belong to the face fan of \(\Delta^{\vee}\). 
Denote by \(F_{\nu}\) the minimal face of \(\Delta^{\vee}\) containing
\(\nu\) and by \(V(F_{\nu})\) the set of all vertices of \(F_{\nu}\).
\begin{lemma}
There exists an \(1\le i\le r\) such that \(V(F_{\nu})\subseteq I_{i}\).
\end{lemma}
\begin{proof}
Since \(I_{1}\sqcup\cdots\sqcup I_{r}=\mathcal{N}(\Delta)(1)\)
is a nef-partition representing the 
Minkowski sum decomposition \(\Delta=\Delta_{1}+\cdots+\Delta_{r}\), the Weil divisor
\begin{equation*}
\sum_{\rho \in I_{i}} D_{\rho}
\end{equation*}
is Cartier on \(X_{\mathcal{N}(\Delta)}\) and
it follows that for any \(\tau\in\mathcal{N}(\Delta)\) 
there exists an \(m_{\tau}\in M\) such that
\begin{equation*}
\begin{cases}
\langle m_{\tau},\rho\rangle = -1, & \mbox{if \(\rho\in\tau(1)\cap I_{i}\)}\\
\langle m_{\tau},\rho\rangle = 0, & \mbox{if \(\rho\in\tau(1)\) but \(\rho\notin I_{i}\)}.
\end{cases}
\end{equation*}
Note that \(F_{\nu}\) is contained in a facet of \(\Delta^{\vee}\).
We may assume that there exists an \(m\in M\) such that
\begin{equation*}
F_{\nu}\subseteq \{n\in N_{\mathbb{R}}~|~\langle m,n\rangle=-1\}.
\end{equation*}
On one hand,
if \(\nu = \sum_{v\in V(F_{\nu})} c_{v} v\) with \(c_{v}>0\), we have
\begin{equation*}
-1 = \langle m,\nu \rangle = 
\sum_{v\in V(F_{\nu})} c_{v}\langle m,v\rangle = - \sum_{v\in V(F_{\nu})} c_{v}.
\end{equation*}
On the other hand, suppose \(V(F_{\nu})\nsubseteq I_{i}\) for all \(i\).
Let \(\tau\) be the cone over \(F_{\nu}\). By the discussion above, for each \(i\),
there exists an \(m^{i}_{\tau}\in M\) such that
\begin{equation*}
\begin{cases}
\langle m^{i}_{\tau},v\rangle = -1, & \mbox{if \(v\in\tau(1)\cap I_{i}\)}\\
\langle m^{i}_{\tau},v\rangle = 0, & \mbox{if \(v\in\tau(1)\) but \(v\notin I_{i}\)}.
\end{cases}
\end{equation*}
By hypothesis, we can always find distinct \(i,j\in\{1,\ldots,r\}\)
so that \(V(F_{\nu})\cap I_{i}\ne\emptyset\)
and \(V(F_{\nu})\cap I_{j}\ne\emptyset\). 
Then we have
\begin{equation*}
0>\langle m_{\tau}^{i},\nu\rangle = \sum_{v\in V(F_{\nu})} c_{v} \langle m^{i}_{\tau},v\rangle
>\sum_{v\in V(F_{\nu})} c_{v} = -1.
\end{equation*}
Here the first inequality holds since \(V(F_{\nu})\cap I_{i}\ne\emptyset\),
and the second one holds since \(V(F_{\nu})\cap I_{j}\ne\emptyset\). 
This contradicts to the fact that \(\langle m_{\tau},\nu\rangle\in\mathbb{Z}\).
\end{proof}

Let \(J_{1}\sqcup\cdots\sqcup J_{r}=\Sigma(1)\) be the nef-partition on \(X\)
induced by the pullback of \(I_{1}\sqcup\cdots\sqcup I_{r}=\mathcal{N}(\Delta)(1)\).
We immediately have the following corollary.
\begin{corollary}
\label{cor:sec-notation}
We have \(J_{i} = \nabla_{i}\cap N\setminus \{\mathbf{0}\}\).
Therefore,
\begin{equation*}
\mathrm{H}^{0}(X^{\vee},\mathscr{O}(F_{k}))\cong \bigoplus_{0\le j\le m_{k}}
\mathbb{C} t^{\rho_{k,j}}
\end{equation*}
and \(W_{i}^{\vee}\) in \S\ref{sit:notation} (2)
is identified with \(\mathrm{H}^{0}(X^{\vee},\mathscr{O}(F_{k}))\). 
In particular,
\begin{equation*}
N\cap \Delta^{\vee}= \bigcup_{i=1}^{r} (\nabla_{i}\cap N).
\end{equation*}
In other words, any integral point in \(\Delta^{\vee}\)
belongs to \(\nabla_{i}\) for some \(1\le i\le r\). 
\end{corollary}

Now we explain how the cycle we integrate over in the period integrals
is related to a cycle in \(\mathrm{H}_{n}(U_{x},\mathscr{E}_{x})\).
Recall that \(X^{\vee}\to\mathbf{P}_{\nabla}\) is
a smooth MPCP desingularization. 
Pick \(x=(s_{1,2},\ldots,s_{r,2})\in W_{1}^{\vee}\times\cdots\times W_{r}^{\vee}\) and 
let \(\pi\colon Y^{\vee}\to X^{\vee}\) be the double cover representing \(x\). 
Denote
by \(R_{x}\) the branch divisor of \(\pi\). Then \(U_{x}=X^{\vee}\setminus R_{x}\). Put
\(V_{x}=Y^{\vee}\setminus R_{x}\) and
let \(i\colon R_{x}\to Y^{\vee}\) and \(j\colon V_{x}\to Y^{\vee}\).
We have a distinguished triangle (indeed a short exact sequence)
\begin{equation}
\label{eq:distinguished-tri}
j_{!} j^{-1} \mathbb{C}_{Y^{\vee}} \to 
\mathbb{C}_{Y^{\vee}} \to
i_{\ast} i^{-1} \mathbb{C}_{Y^{\vee}}\to.
\end{equation}
Applying the functor \(R\pi_{\ast}=R\pi_{!}\) to \eqref{eq:distinguished-tri}
and the base change theorem from the 
commutative diagram
\begin{equation*}
\begin{tikzcd}
& V_{x}\ar[r,"j"]\ar[d,"\pi"] & Y^{\vee}\ar[d,"\pi"] & R_{x}\ar[l,"i"]\ar[d,equal]\\
& U_{x}\ar[r,"j"] & X^{\vee} & R_{x}\ar[l,"i"]
\end{tikzcd},
\end{equation*}
we arrive at a triangle
\begin{equation*}
j_{!}R\pi_{\ast}\mathbb{C}_{V_{x}}\to R\pi_{\ast}\mathbb{C}_{Y^{\vee}}
\to i_{\ast}\mathbb{C}_{R_{x}}\to.
\end{equation*}
\(R\pi_{\ast}\mathbb{C}_{V_{x}}\) (resp.~\(R\pi_{\ast}\mathbb{C}_{Y^{\vee}}\))
is decomposed into eigensubsheaves \(\mathbb{C}_{U_{x}}\oplus\mathscr{E}_{x}\)
(resp.~\(\mathbb{C}_{X^{\vee}}\oplus\mathscr{G}\))
according to the Galois group action of the cover \(Y^{\vee}\to X^{\vee}\). 
Moreover, \(j_{!}\mathscr{E}_{x}=\mathscr{G}\).
Taking \(R\Gamma_{\mathrm{c}}(X,-)\), we see that
\begin{equation*}
\mathrm{H}_{\mathrm{c}}^{n}(Y^{\vee},\mathbb{C})\cong 
\mathrm{H}_{\mathrm{c}}^{n}(X^{\vee},\mathbb{C})\oplus
\mathrm{H}_{\mathrm{c}}^{n}(U_{x},\mathscr{E}_{x}).
\end{equation*}
Combined with Poincar\'{e} duality, 
since the integration over cycles in 
\(\mathrm{H}_{\mathrm{c}}^{n}(X^{\vee},\mathbb{C})\cong
\mathrm{H}_{n}(X^{\vee},\mathbb{C})\),
we conclude that
the affine period integrals defined in Definition 
\ref{def:affine-perid} are exactly the 
restriction of all possibly non-trivial
period integrals of \(\mathcal{Y}^{\vee}\to V\) to \(T\).
Also note that when \(n\) is odd, 
we have \(\mathrm{H}_{\mathrm{c}}^{n}(Y^{\vee},\mathbb{C})\cong 
\mathrm{H}_{\mathrm{c}}^{n}(U_{x},\mathscr{E}_{x})\cong\mathrm{H}_{n}(U_{x},\mathscr{E}_{x})\).

\section{\texorpdfstring{\(\beta\)}{beta} is non-resonant}
\label{sec:non-resonant}
We continue to assume that \(A\) and \(\beta\)
are the same as in \S\ref{sit:notation} (3) in this section.
In this section, we show that in the present case
the exponent \(\beta\) is non-resonant with respect to \(A\). 
Combined with Theorem 
\ref{thm:gkz-non-resonant}, we obtain Theorem \ref{thm:main-theorem}.

The defining fan \(\Sigma\) of \(X\) is a MPCP desingularization of \(\mathbf{P}_{\Delta}\),
the toric variety defined by 
\(\mathcal{F}(\Delta^{\vee})\), the 
\emph{face fan} of the Cayley polytope
\(\operatorname{Conv}(\nabla_{1},\ldots,\nabla_{r})=\Delta^{\vee}\)
(or equivalently a MPCP desingularization of 
\(\mathcal{N}(\Delta)\), the \emph{normal fan} of \(\Delta\)).
Recall that by Batyrev--Borisov's duality construction, if 
\(I_{1}\sqcup\cdots\sqcup I_{r}=\mathcal{N}(\Delta)(1)\)
is the nef-partition on \(\mathbf{P}_{\Delta}\)
representing the Minkowski sum decomposition 
\(\Delta=\Delta_{1}+\cdots+\Delta_{r}\), then
\(\nabla_{i}=\operatorname{Conv}(I_{i}\cup\{\mathbf{0}\})\).

For each \(\sigma\in \Sigma(n)\), let
\begin{equation*}
\hat{\sigma}:=\operatorname{Cone}(\{\hat{\rho}~|~
\rho\in\sigma(1)\}\cup\{e_{1}\times\mathbf{0},
\ldots,e_{r}\times\mathbf{0}\})
\end{equation*}
and
\begin{equation*}
\operatorname{Poly}(\hat{\sigma}):=\operatorname{Conv}(\{\hat{\rho}~|~
\rho\in\sigma(1)\}\cup\{e_{1}\times\mathbf{0},
\ldots,e_{r}\times\mathbf{0}\})
\end{equation*}
where \(\hat{\rho}=(e_{i},\rho)\) if \(\rho\in J_{i}\).
Under our notation in \S\ref{sit:notation}, we have
\(\hat{\rho}_{i,j} = \mu_{i,j}\).
Then \(\operatorname{Poly}(\hat{\sigma})\) is a convex polytope in 
\(\mathbb{R}^{r}\times N_{\mathbb{R}}\). 
Moreover, owing to our hypothesis on \(X\), 
\(\operatorname{Poly}(\hat{\sigma})\) is indeed a simplex 
(its normalized volume is \(1\)). Recall that 
\begin{equation*}
E_{i}:=\sum_{\rho\in J_{i}} D_{\rho}.
\end{equation*}
is a numerically effective divisor on \(X\);
the support function \(\varphi_{i}\) of \(E_{i}\) is convex.
We have the following observation.
\begin{lemma}
\label{lem:lifting-convex}
Let \(\sigma\in\Sigma(n)\) and \(\rho_{1},\ldots,\rho_{n}\) be
the primitive generator of the \(1\)-cones in \(\sigma\).
For non-negative scalars \(u_{1},\ldots,u_{n}\), we have
\begin{equation}
\sum_{k=1}^{n} u_{k} \hat{\rho}_{k} = 
\begin{bmatrix}
-\varphi_{1}(\sum_{k=1}^{n} u_{k}\rho_{k})\\
\vdots\\
-\varphi_{r}(\sum_{k=1}^{n} u_{k}\rho_{k})\\
\vline height 1ex \\
\sum_{k=1}^{n} u_{k}\rho_{k}\\
\vline height 1ex \\
\end{bmatrix}\in\mathbb{R}^{r}\times N_{\mathbb{R}}.
\end{equation}
\end{lemma}
\begin{proof}
Let \(m_{\sigma}^{i}\in M\) be the element defining the support function
\(\varphi_{i}\) on \(\sigma\),
i.e.,~\(\varphi_{i}(v)=\langle v,m_{\sigma}^{i}\rangle\) for \(v\in \sigma\).
Then 
\begin{equation}
\varphi_{i}(\rho_{j})=\langle m_{\sigma}^{i},\rho_{j}\rangle=
\begin{cases}
-1 & \mbox{if \(\rho_{j}\in J_{i}\)}\\
0 & \mbox{if \(\rho_{j}\notin J_{i}\)}.
\end{cases}
\end{equation}
Said differently, we have
\begin{equation}
\hat{\rho}_{j}=
\begin{bmatrix}
-\varphi_{1}(\rho_{j})\\
\vdots\\
-\varphi_{r}(\rho_{j})\\
\vline height 1ex \\
\rho_{j}\\
\vline height 1ex \\
\end{bmatrix}\in \mathbb{Z}^{r}\times N.
\end{equation}
Since \(\varphi_{i}\) is linear on \(\sigma\), the result follows.
\end{proof}

\begin{proposition}
We have
\begin{equation}
\label{eq:union-cones}
\operatorname{Conv}(A\cup\{\mathbf{0}\}) = \bigcup_{\sigma\in\Sigma(n)}
\operatorname{Poly}(\hat{\sigma}).
\end{equation}
\end{proposition}
\begin{proof}
The inclusion ``\(\supseteq\)'' is clear.
To establish the reverse inclusion, it suffices to show that
the right hand side of \eqref{eq:union-cones}, which
is denoted by \(\mathcal{B}\) in the sequel, is convex.
Pick \(\nu_{0}\) and \(\nu_{1}\) from \(\mathcal{B}\).
We shall prove that the line segment connecting \(\nu_{0}\) and \(\nu_{1}\)
is also contained in \(\mathcal{B}\).
We may assume \(\nu_{j}\in\operatorname{Poly}
(\hat{\sigma}_{j})\) for some \(\sigma_{j}\in\Sigma(n)\).
Denote by \(\rho^{j}_{1},\ldots,\rho^{j}_{n}\)
the primitive generator of the \(1\)-cones contained in \(\sigma_{j}\).
Write
\begin{equation*}
\nu_{j} = \sum_{k=1}^{n} c_{j,k}\hat{\rho}^{j}_{k} + 
\sum_{l=1}^{r} d_{j,l} (e_{l}\times\mathbf{0})
\end{equation*}
with \(c_{j,k},d_{j,l}\ge 0\) and \(\sum_{k=1}^{n} c_{j,k}+\sum_{l=1}^{r} d_{j,l}\le 1\).

Our goal is to prove that \(\nu_{t}:=t\nu_{1}+(1-t)\nu_{0}\in\mathcal{B}\)
for all \(t\in [0,1]\). 
First we observe that \(\mathcal{B}\)
projects onto
the Cayley polytope \(\operatorname{Conv}(\nabla_{1},\ldots,\nabla_{r})=\Delta^{\vee}\)
under the canonical projection \(
\varpi\colon \mathbb{R}^{r}\times N_{\mathbb{R}}\to N_{\mathbb{R}}\);
the line segment \(t\nu_{1}+(1-t)\nu_{0}\) projects onto a line segment in \(\Delta^{\vee}\).
Fix \(t\in [0,1]\)
and put \(\nu\equiv\nu_{t}\). Let \(\sigma\in\Sigma(n)\) be the maximal cone
containing
\(\varpi(\nu)\) and \(\rho_{1},\ldots,\rho_{n}\) be the primitive generator of 
the \(1\)-cones in \(\sigma\).
We shall now prove that 
\(\nu\in\operatorname{Poly}(\hat{\sigma})\), i.e.,~we have
to prove that
\begin{itemize}
\item[(1)] There are non-negative scalars \(c_{1},\ldots,c_{n}\) and 
\(d_{1},\ldots,d_{r}\) such that
\begin{equation*}
\nu = \sum_{k=1}^{n} c_{k} \hat{\rho}_{k} + \sum_{l=1}^{r} d_{l} (e_{l}\times\mathbf{0});
\end{equation*}
\item[(2)] \(\sum_{k=1}^{n} c_{k}+\sum_{l=1}^{r} d_{l}\le 1\).
\end{itemize}

Let us prove (1). Since \(\varpi(\nu)\in\Delta^{\vee}\cap\sigma\), there 
are non-negative scalars \(c_{1},\ldots,c_{n}\) such that 
\begin{equation*}
\varpi(\nu) = \sum_{k=1}^{n} c_{k}\rho_{k}.
\end{equation*}
By Lemma \ref{lem:lifting-convex},
\begin{align*}
\nu - \sum_{k=1}^{n}c_{k}\hat{\rho}_{k}=\nu -
\begin{bmatrix}
-\varphi_{1}(\sum_{k=1}^{n} c_{k}\rho_{k})\\
\vdots\\
-\varphi_{r}(\sum_{k=1}^{n} c_{k}\rho_{k})\\
\vline height 1ex \\
\sum_{k=1}^{n} c_{k}\rho_{k}\\
\vline height 1ex \\
\end{bmatrix}=
\begin{bmatrix}
d_{1}\\
\vdots\\
d_{r}\\
\vline height 1ex\\
\mathbf{0}\\
\vline height 1ex\\
\end{bmatrix}
\end{align*}
with \(d_{i}=-t\varphi_{i}(\sum_{k=1}^{n}c_{1,k}\rho_{k}^{1})-(1-t)
\varphi_{i}(\sum_{k=1}^{n}c_{0,k}\rho_{k}^{0})+\varphi_{i}(\sum_{k=1}^{n} c_{k}\rho_{k})\).
Moreover, \(d_{i}\ge 0\) since
\(\varpi(\nu)=t\varpi(\nu_{1})+(1-t)\varpi(\nu_{0})\)
and \(\varphi_{i}\) is convex.

Now we prove (2). Since \(J_{1}\sqcup\cdots\sqcup J_{r}=\Sigma(1)\) is a nef-partition,
for each \(1\le j\le n\) there exists exactly one \(i\in\{1,\ldots,r\}\)
satisfying \(\varphi_{i}(\rho_{j})=-1\). Therefore
\begin{equation*}
\sum_{k=1}^{n}c_{k}+\sum_{l=1}^{r}d_{l} = 
\sum_{l=1}^{r} (-\varphi_{l}(\textstyle\sum_{k=1}^{n} c_{k}\rho_{k})+d_{l})
=\langle \vec{n},\nu\rangle,
\end{equation*}
where
\begin{equation*}
\vec{n}=
\begin{bmatrix}
1\\
\vdots\\
1\\
\vline height 1ex\\
\mathbf{0}\\
\vline height 1ex\\
\end{bmatrix}\in\mathbb{Z}^{r}\times M.
\end{equation*}
Now \(\langle\vec{n},\nu_{j}\rangle
=\sum_{k=1}^{n}c_{j,k}+\sum_{l=1}^{r}d_{j,l}\le 1\) for \(j=0,1\).
It follows from the linearity that 
\(\langle \vec{n},\nu\rangle=
\langle \vec{n},t\nu_{1}+(1-t)\nu_{0}\rangle\le 1\). This establishes (2).
\end{proof}

\begin{corollary}
\label{cor:volume}
The normalized volume of \(\operatorname{Conv}(A\cup\{\mathbf{0}\})\)
is equal to \(|\Sigma(n)|\), the number of maximal cones in \(\Sigma\).
Therefore, the holonomic rank of \(\mathcal{M}_{A}^{\beta}\)
is \(|\Sigma(n)|\).
\end{corollary}
\begin{proof}
Since \(X\) is smooth by our hypothesis, each cone \(\sigma\)
has normalized volume \(1\) and so does \(\hat{\sigma}\).
\end{proof}

Next we examine the facets of \(\mathbb{R}_{+}A\).
\begin{proposition}
\label{prop:cone-A}
\begin{equation*}
\mathbb{R}_{+}A = \bigcup_{\sigma\in\Sigma(n)} \hat{\sigma}.
\end{equation*}
\end{proposition}
\begin{proof}
The direction ``\(\supseteq\)'' is clear. Now we prove the converse.
Let 
\begin{equation*}
\nu:=\sum_{\rho\in \Sigma(1)} c_{\rho} \hat{\rho} + 
\sum_{l=1}^{r} d_{l} (e_{l}\times\mathbf{0})
\quad~\mbox{with~\(c_{\rho},d_{l}\ge 0\).}
\end{equation*}
Here by abuse of notation we denote the primitive 
generator of \(\rho\in \Sigma(1)\) by \(\rho\).

Since \(\Sigma\) is proper, the vector \(\sum_{\rho\in\Sigma(1)} c_{\rho}\rho \)
is contained in some maximal cone \(\sigma\in\Sigma(n)\). Let 
\(\rho_{1},\ldots,\rho_{n}\) be the primitive generator of the \(1\)-cones
in \(\sigma\) as before. There are non-negative scalars \(c_{1},\ldots,c_{n}\) such that
\begin{equation*}
\sum_{\rho\in\Sigma(1)} c_{\rho}\rho = \sum_{k=1}^{n} c_{k} \rho_{k}.
\end{equation*}
Now we compare \(\sum_{\rho\in\Sigma(1)} c_{\rho}\hat{\rho}\)
with \(\sum_{k=1}^{n} c_{k} \hat{\rho}_{k}\).
Notice that by the convexity of \(\varphi_{i}\) 
\begin{equation*}
\sum_{\rho\in\Sigma(1)}-c_{\rho}\varphi_{i}(\rho)\ge
-\varphi_{i}(\textstyle \sum_{\rho\in\Sigma(1)}c_{\rho}\rho)
=-\varphi_{i}(\textstyle \sum_{k=1}^{n} c_{k} \rho_{k}).
\end{equation*}
It the follows that
\begin{equation*}
\sum_{\rho\in\Sigma(1)} c_{\rho}\hat{\rho} - 
\sum_{k=1}^{n} c_{k} \hat{\rho}_{k}\in \operatorname{Cone}(\{e_{1}\times\mathbf{0},\ldots,
e_{r}\times\mathbf{0}\}).
\end{equation*}
Consequently, \(\nu\in\hat{\sigma}\) as desired.
\end{proof}

Let \(F\) be a facet of \(\mathbb{R}_{+}A\). By Proposition \ref{prop:cone-A},
\(F\) must be a facet of \(\hat{\sigma}\) for some \(\sigma\in\Sigma(n)\).
Let again \(\rho_{1},\ldots,\rho_{n}\) be the primitive generator
of the \(1\)-cones in \(\sigma\).
Since \(\hat{\sigma}\) is simplicial, \(F\) must contain all but one vectors in
\begin{equation*}
\{\hat{\rho}_{1},\ldots,\hat{\rho}_{n},e_{1}\times\mathbf{0},\ldots,e_{r}\times\mathbf{0}\}.
\end{equation*}
We claim that \(F\) can not contain all \(e_{i}\times\mathbf{0}\).
Otherwise, if \(F\) is defined by \((a,m)\in\mathbb{R}^{r}\times M_{\mathbb{R}}\)
(i.e.,~\(F=\mathbb{R}_{+}A\cap \{(b,n)\in\mathbb{R}^{r}\times N_{\mathbb{R}}
~|~\langle (a,m),(b,n)\rangle=0\}\)),
we must have \(a=0\). Because \(\Sigma\) is proper, \(\mathbb{R}_{+}A\)
can not fall into the closed half-space defined by \((0,m)\) and 
\(F\) can not be a face. This is 
a contradiction. 

Assume that \(e_{j}\times\mathbf{0}\) is omitted. It follows that if \(F\) is defined by 
\((a,m)\in\mathbb{R}^{r}\times M_{\mathbb{R}}\) as above, we have
\(a_{i}=0\) for \(i\ne j\). Moreover, we may assume \(a_{j}=1\) 
(and hence \(a=e_{j}\)). Then \(m\)
satisfies the equations \(\langle \hat{\rho}_{k},(e_{j},m)\rangle = 0\), i.e.,
\begin{equation*}
\begin{cases}
\langle m,\rho_{k}\rangle = 0 & \mbox{if \(\rho_{k}\notin I_{j}\)}\\
\langle m,\rho_{k}\rangle = -1 & \mbox{if \(\rho_{k}\in I_{j}\)}.
\end{cases}
\end{equation*}
We see that \(m = m_{\sigma}^{j}\), the Cartier data of \(E_{j}\) on \(\sigma\).
In particular, \(m\in M\).

\begin{theorem}
\label{thm:non-resonant}
Let \(A\) and \(\beta\) be as in 
\S\ref{sit:notation} (3). Then
\(\beta\) is non-resonant, i.e.,
\begin{equation*}
\beta\notin \bigcup_{F\subset \mathbb{R}_{+}A} 
\left(\mathbb{C}F+\mathbb{Z}^{r+n}\right)
\end{equation*}
where the union runs through all the proper faces of \(\mathbb{R}_{+}A\).
\end{theorem}
\begin{proof}
Suppose \(\beta\in \mathbb{C}F+\mathbb{Z}^{r+n}\) for some facet \(F\).
Write \(\beta = f + z\) with \(f\in \mathbb{C}F\) and \(z\in \mathbb{Z}^{r+n}\).
According to our discussion above, \(F\) is defined by 
an element of the form \((e_{j},m)\) and \(m\in M\). But then
\begin{equation*}
\frac{1}{2}=\langle (e_{j},m),\beta\rangle =
\langle (e_{j},m),f+z\rangle=\langle (e_{j},m),z\rangle\in\mathbb{Z}
\end{equation*}
which is absurd.
\end{proof}

\begin{proof}[Proof of Theorem \ref{thm:main-theorem}]
By Theorem \ref{thm:non-resonant}, \(\beta\) 
is non-resonant with respect to \(A\). Theorem
\ref{thm:main-theorem} is now a direct consequence of
Theorem \ref{thm:gkz-non-resonant}.
\end{proof}

\begin{corollary}
\(\mathcal{M}_{A}^{\beta}\) admits a rank one point, i.e.,~there 
is a \(x\in W_{1}^{\vee}\times\cdots\times W_{r}^{\vee}\) such that
\(\operatorname{Sol}^{0}(\mathcal{M}_{A}^{\beta})_{x}\) is \(1\)-dimensional.
\end{corollary}
\begin{proof}
It suffices to pick \(x=(1,\ldots,1)\in W_{1}^{\vee}\times\cdots\times W_{r}^{\vee}\).
In which case, \(U_{x}=T\) and \(\mathscr{E}_{x}=\mathbb{C}\) is the constant sheaf.
Therefore, \(\mathrm{H}_{n}(U_{x},\mathscr{E}_{x})\cong\mathbb{C}\).
\end{proof}

\begin{remark}
It is clear from the proof that the assumption made in 
\S\ref{subsec:cy-fci} can be weakened. We only need to assume that
\(\mathbf{P}_{\Delta}\) admits a MPCP desingularization \(X_{\Sigma}\)
such that \(\Sigma(1)\) generates \(\mathbb{Z}^{n}\) as an abelian group;
this assumption allows us to apply the result 
in \cite{1990-Gelfand-Kapranov-Zelevinsky-generalized-euler-integrals-and-a-hypergeometric-functions}.
Under this weakened assumption, all the results in \S\ref{sec:non-resonant} 
(except for Corollary \ref{cor:volume}) are still valid
and the chain-integral map \eqref{eq:chain-integral-map} is an isomorphism.
\end{remark}

Using the generalized Frobenius method
in \cite{2022-Lee-Lian-Yau-on-calabi-yau-fractional-complete-intersections}, 
we can write down explicitly all
the power series solutions to \(\mathcal{M}_{A}^{\beta}\).

Put \(D_{i,j}=D_{\rho_{i,j}}\) for \(1\le i\le r\) and \(1\le j\le m_{i}\)
and \(D_{i,0}:=-\sum_{j=1}^{m_{i}} D_{i,j}\). We can
define a cohomology-valued series
\begin{equation*}
B_{X}^{\alpha}(x):=\left(\sum_{\ell\in \overline{\mathrm{NE}}(X)\cap \operatorname{ker}(A)} 
\mathcal{O}_{\ell}^{\alpha} x^{\ell+\alpha}\right)
\exp\left(\sum_{i=1}^{r}\sum_{j=0}^{m_i}(\log x_{i,j}) D_{i,j}\right).
\end{equation*}
Here we think of \(A\) as a linear map \(A\colon \mathbb{Z}^{r+p}\to\mathbb{Z}^{r+n}\)
and identify the Mori cone \(\overline{\mathrm{NE}}(X)\) of 
\(X\) with a cone in \(\mathbb{R}^{r+p}\). 
Similar to \S\ref{sit:notation}~(3), the components of elements
\(\ell\in\operatorname{ker}(A)\) are labeled by \((i,j)\).
Finally
\begin{equation*}
\mathcal{O}^\alpha_\ell
:=\frac{\prod_{i=1}^{r}(-1)^{\ell_{i,0}}
\Gamma(-D_{i,0}-\ell_{i,0}-\alpha_{i,0})}{\prod_{i=1}^{r}\Gamma(-\alpha_{i,0})
\prod_{i=1}^r\prod_{j=1}^{n_i}\Gamma(D_{i,j}+\ell_{i,j}+\alpha_{i,j}+1)}.
\end{equation*}
and \(\alpha=(\alpha_{i,j})\in\mathbb{Q}^{r+p}\)
with \(\alpha_{i,0}=-1/2\) and \(\alpha_{i,j}=0\) for \(1\le j\le m_{i}\).
We refer the reader to \cite{2022-Lee-Lian-Yau-on-calabi-yau-fractional-complete-intersections}*{\S2} for detailed explanations.

The element \(B^{\alpha}_{X}(x)\) is understood as a
\(\mathrm{H}^{\bullet}(X,\mathbb{C})\)-valued function.
Note that \(\dim \mathrm{H}^{\bullet}(X,\mathbb{C})=|\Sigma(n)|\)
is equal to the normalized volume of \(\operatorname{Conv}(A\cup\{\mathbf{0}\})\)
by Corollary \ref{cor:volume} and
it is shown that the coefficients of \(B_{X}^{\alpha}(x)\) form a basis
of the set of solutions to \(\mathcal{M}_{A}^{\beta}\)
(cf.~\cite{2022-Lee-Lian-Yau-on-calabi-yau-fractional-complete-intersections}*{Corollary 3.4}).
Combined with Theorem \ref{thm:theorem-a}, we obtain
\begin{corollary}
\label{cor:periods-power-series}
The coefficients of the vector-valued function \(B_{X}^{\alpha}(x)\)
form a basis of the set of period integrals for \(\mathcal{Y}^{\vee}\to V\).
\end{corollary}

\begin{instance}
\label{ex:p3-cover}
We can apply our results to cyclic covers of \(X^{\vee}=\mathbf{P}^{3}\)
branch over eight hyperplanes in general position.
It is known that such a Calabi--Yau double cover \(Y^{\vee}\) admits a
crepant resolution \(\tilde{Y}^{\vee}\to Y^{\vee}\) and its middle cohomology has
Hodge numbers \((1,9,9,1)\).
Moreover, it is proven that the moduli space of the 
complex deformation of \(\tilde{Y}^{\vee}\) can not be embedded as a Zariski
open subset of a locally hermitian symmetric domain by the period map 
\cite{2013-Gerkmann-Sheng-van-Straten-Zuo-on-the-monodromy-of-the-moduli-space-of-calabi-yau-threefolds-coming-from-eight-planes-in-p3}. 

In this case, we have
\begin{equation*}
\begin{cases}
\nabla_{1} = \operatorname{Conv}(\{\mathbf{0},(1,0,0),(0,1,0),(0,0,1)\}),\\
\nabla_{2} = \operatorname{Conv}(\{\mathbf{0},(-1,0,0),(-1,1,0),(-1,0,1)\}),\\
\nabla_{3} = \operatorname{Conv}(\{\mathbf{0},(0,-1,0),(1,-1,0),(0,-1,1)\}),\\
\nabla_{4} = \operatorname{Conv}(\{\mathbf{0},(0,0,-1),(1,0,-1),(0,1,-1)\})
\end{cases}
\end{equation*}
and \(\nabla=\nabla_{1}+\cdots+\nabla_{4}\). One can check that 
\(\mathbf{P}_{\nabla}=\mathbf{P}^{3}=X^{\vee}\).
Now it is straightforward to check that \(\mathbf{P}_{\Delta}\) with
\(\Delta^{\vee}=\operatorname{Conv}(\nabla_{1},\ldots,\nabla_{4})\)
admits a smooth MPCP desingularization \(X\to \mathbf{P}_{\Delta}\).
Applying our results, the solutions to \(\mathcal{M}_{A}^{\beta}\)
with
\begin{align*}
A = 
\begin{bmatrix}
1 & 1 & 1 & 1 & 0 & 0 & 0 & 0 & 0 & 0 & 0 & 0 & 0 & 0 & 0 & 0\\
0 & 0 & 0 & 0 & 1 & 1 & 1 & 1 & 0 & 0 & 0 & 0 & 0 & 0 & 0 & 0\\
0 & 0 & 0 & 0 & 0 & 0 & 0 & 0 & 1 & 1 & 1 & 1 & 0 & 0 & 0 & 0\\
0 & 0 & 0 & 0 & 0 & 0 & 0 & 0 & 0 & 0 & 0 & 0 & 1 & 1 & 1 & 1\\
0 & 1 & 0 & 0 & 0 &-1 &-1 &-1 & 0 & 0 & 1 & 0 & 0 & 0 & 1 & 0\\
0 & 0 & 1 & 0 & 0 & 0 & 1 & 0 & 0 &-1 &-1 &-1 & 0 & 0 & 0 & 1\\
0 & 0 & 0 & 1 & 0 & 0 & 0 & 1 & 0 & 0 & 0 & 1 & 0 &-1 &-1 &-1\\
\end{bmatrix},~
\beta=
\begin{bmatrix}
-1/2\\
-1/2\\
-1/2\\
-1/2\\
0\\
0\\
0
\end{bmatrix}
\end{align*}
are precisely the period integrals of 
the family \(\mathcal{Y}^{\vee}\to V\)
of Calabi--Yau fractional complete
intersections and 
\(B_{X}^{\alpha}(x)\) gives the power series expansion
of the period integrals.
\end{instance}


\end{document}